\documentclass[french,leqno]{smfart}
\usepackage{amsmath,amstext, amsthm, amssymb}

\usepackage{pifont}
\usepackage[hypertex]{hyperref}
\usepackage{srcltx}
\usepackage[french]{babel}
\usepackage[all]{xy}
\usepackage{color}
\usepackage{t1enc}
\usepackage{smfthm}

\numberwithin{equation}{section}

\def\beq{\begin{equation}}
\def\eeq{\end{equation}}

\def\crash#1{}

\def\Z{{\mathbb Z}}
\def\Q{{\mathbb Q}}

\def\C{{\mathbb C}}

\def\O{\Omega}

\def\l{\left}
\def\r{\right}
\def\[[{\l[\l[}
\def\]]{\r]\r]}
\def\p{\prime}

\def\sgq{\sigma_q}

\def\cf{\emph{cf. }}
\def\ie{\emph{i.e. }}
\def\ds{\displaystyle}

\def\cC{{\mathcal C}}

\def\cE{{\mathcal E}}
\def\cF{{\mathcal F}}
\def\cK{{\mathcal K}}
\def\cM{{\mathcal M}}

\def\cL{{\mathcal{ L}}}

\def\cP{{\mathcal P}}

\def\cR{{\mathcal R}}

\def\wtilde{\widetilde}

\def\a{\alpha}

\def\de{\delta}
\def\ga{\gamma}
\def\sg{\sigma}

\def\Ga{\Gamma}
\def\la{\lambda}

\def\Sg{\Sigma}
\def\De{\Delta}

\def\deg{\mathop{\rm deg}}

\author{Lucia Di Vizio}

\date{\today}

\title{Approche galoisienne de la transcendance différentielle}

\begin{document}

\bibliographystyle{plain}

\maketitle

\tableofcontents

\section*{Introduction}

Le théorème de Hölder dit que la fonction Gamma d'Euler n'est pas solution d'une équation différentielle
algébrique à coefficients dans le corps $\C(x)$ des fonctions rationnelles à
coefficients complexes. Une littérature relativement vaste est dédiée à cet énoncé, à la fois dans le but d'en donner
une nouvelle preuve et de le généraliser.
\par
Le théorème de Hölder a été le premier résultat de
\emph{transcendance différentielle} ou d'\emph{hypertranscendance}.
En théorie de la transcendance, on s'intéresse à la transcendance des fonctions sur $\Q(x)$, souvent pour en déduire
la transcendance de leurs valeurs spéciales.
En hypertranscendance on s'intéresse
à la transcendance simultanée des fonctions et de toutes leurs dérivées, donc à la
propriété d'un ensemble de fonctions de ne pas être solution d'une équation différentielle algébrique.
Depuis le théorème de Hölder,  les mathématiciens ont étudié l'hypertranscendance
pour les raisons le plus diverses.
On en évoquera quelques unes tout au long de ce papier, et en particulier dans le dernier paragraphe.
\par
Récemment une théorie de Galois paramétrée des équations fonctionnelles a vu le jour (\cf \cite{cassisinger} et \cite{HardouinSinger}),
avec le bout
de fournir une approche systématique à la transcendance différentielle, en contraste avec la littérature
plus ancienne où on traite les différentes fonctions spéciales ``au cas par cas''.
On va donner un bref survol de cette théorie et en montrer quelques applications en révisitant des
résultats classiques, tel le théorème de Hölder, par exemple.

\medskip
\emph{Remerciements.}
C'est un plaisir de remercier W. Bergweiler, D. Bertrand, X. Buff,
C. Hardouin, B.Q. Li, Pierre Nguyen, M.F. Singer, Z. Ye pour leurs suggestions et commentaires.
Je suis particulièrement reconnaissante à D. Bertrand, Z. Djadli, D. Harari, C. Hardouin,
M.F. Singer pour leur relecture attentive du manuscrit et leur remarques et corrections et a` J.-P. Allouche pour son invitation.
\section{Fonction Gamma et théorème de Hölder}
\label{sec:Holder}

En 1729, Euler, dans une lettre à Goldbach \cite{eulerletter}, définit la fonction Gamma grâce aux limites suivantes,
convergentes pour tout
$x\in\C\smallsetminus\Z_{<0}$:
$$
\begin{array}{rcl}
\Ga(x)
&=&\ds\frac{1}{x}\prod_{n=1}^\infty
\l[\l(1+\frac{1}{n}\r)^x\l(1+\frac{x}{n}\r)^{-1}\r]\\~\\
&=&\ds\lim_{n\to\infty}\frac{(n-1)!}{x(x+1)\cdots(x+n-1)}n^x.
\end{array}
$$
Weierstrass en a donné une autre caractérisation qui montre que $1/\Ga$ est une fonction analytique entière:
$$
\frac{1}{\Ga(x)}=xe^{\ga x}\prod_{n=1}^\infty
\l(1+\frac{x}{n}\r)e^{-\frac x n},
$$
où $\gamma$ est la constante d'Euler-Mascheroni:
$$
\ga=\lim_{m\to\infty}
\l(1+\frac{1}{2}+\dots+\frac{1}{m}-\log m\r)=0,5772157....
$$
Une propriété fondamentale de la fonction $\Gamma$ est celle de vérifier l'équation aux différences
\beq\label{eq:ga}
y(x+1)=xy(x)
\eeq
qui, compte tenu du fait que $\Ga(1)=1$,
implique immédiatement que $\Ga(n)=(n-1)!$, pour tout entier positif $n$.
D'après le théorème de Bohr-Mollerup, la fonction $\Gamma$ est l'unique solution
de l'équation $y(x+1)=xy(x)$, logarithmiquement convexe et telle que y(1)=1.
L'équation \eqref{eq:ga} est linéaire, donc l'ensemble de toutes ses solutions méromorphes
forme un espace vectoriel engendré par $\Ga$ sur le corps des fonctions $1$-périodiques.
\footnote{Pour une présentation des formules classiques sur la fonction
Gamma et pour des références précises à la littérature plus ancienne, voir \cite[\S XII]{whittakerwatson}.}
\par
Le célèbre théorème de Hölder \cite{holderGamma} affirme que
la fonction $\Gamma$ est \emph{différentiellement transcendante} (ou \emph{hypertranscendante}) sur $\C(x)$,
c'est-à-dire:

\begin{theo}
La fonction $\Gamma$ n'est pas solution d'une équation différentielle algébrique à
coefficients dans $\C(x)$.
\end{theo}

Il existe nombreuses preuves de ce résultat (voir par exemple
\cite{owstroskiGamma1}, \cite{owstroskiGamma2}, \cite{hausdorffGamma}, \cite{BankKaufamannGamma2},
\cite{MooreGamma}, \cite{Nielsen}) et aussi
des nombreuses généralisations dans des directions différentes
(voir par exemple \cite{BankKaufamanGamma}, \cite{BankGammaHyptrFunctions}, \cite{BankKaufamnnGammaNevalinna}, \cite{GrossOsgood},
\cite{Nishizawa1}, \cite{Nishizawa2}, \cite{Miller}, \cite{Pastro}).
Parmi ces démonstrations on citera celle de Bank et Kaufmann \cite{BankKaufamannGamma2}, qui se déduit du théorème suivant:

\begin{theo}\label{theo:BankKaukmannGamma}
Soit $\cF$ un sous-corps du corps $\cM er(\C)$
des fonctions méromorphes sur $\C$, contenant $\C(x)$ et fermé par rapport
à l'opérateur de translation $\tau:f(x)\mapsto f(x+1)$ et à la dérivation par rapport à $x$.
Si la fonction $\Gamma$ d'Euler est solution d'une équation différentielle algébrique
à coefficients dans $\cF$, alors, il existe
$g,f_0,f_1,\dots,f_n\in \cF$, avec $f_0,f_1,\dots,f_n$ périodiques de période $1$,
non tous nuls, telles que
$$
\sum_{i=0}^n f_i(x)\frac{d^i}{dx^i}\l(\frac{1}{x}\r)=g(x)-g(x+1).
$$
\end{theo}

\begin{rema}\label{rema:implicationfacile}
La réciproque de cet énoncé est quasiment vraie.
Soit $\cF\langle\Gamma(x)\rangle_{\frac{d}{dx}}$
la plus petite extension de $\cF(\Gamma(x))$ contenue dans $\cM er(\C)$, fermée par rapport à la dérivation
$\frac{d}{dx}$.
Si on suppose que le corps différentiel $\cF\langle\Gamma(x)\rangle_{\frac{d}{dx}}$
ne contient pas plus de fonctions périodiques que
$\cF$, alors la réciproque du théorème ci-dessus est immédiate.
En effet, soit $\psi(x):=\frac{\Gamma^\p(x)}{\Gamma(x)}$ la dérivée logarithmique de $\Gamma(x)$,
qu'on appelle usuellement \emph{fonction digamma}.
On a:
$$
\begin{array}{rcl}
\lefteqn{\tau\l(\sum_{i=0}^n f_i(x)\frac{d^i\psi}{dx^i}(x)+g(x)\r)}\\
&=&\ds\sum_{i=0}^n f_i(x)\frac{d^i\psi}{dx^i}(x)
+\sum_{i=0}^n f_i(x)\frac{d^i}{dx^i}\l(\frac{1}{x}\r)+g(x+1)\\
&=&\ds\sum_{i=0}^n f_i(x)\frac{d^i\psi}{dx^i}(x)+g(x).
\end{array}
$$
Ceci implique que $\sum_{i=0}^n f_i(x)\frac{d^i\psi}{dx^i}(x)+g(x)$ est une fonction
périodique de $\cF$ et nous fournit gratuitement une relation différentielle algébrique sur $\cF$
pour $\psi$, et, donc, pour $\Gamma$.
\end{rema}

La démonstration du Théorème \ref{theo:BankKaukmannGamma} donnée  dans \cite{BankKaufamannGamma2} est assez élémentaire.
On en donnera une preuve galoisienne plus loin. Le théorème de Hölder s'en déduit aisément en raisonnant
sur les pôles de $g(x)-g(x+1)$,
compte tenu du fait que les seules fonctions périodiques contenues dans $\C(x)$ sont les constates.
\par
Notons que, comme dans le cas des extensions algébriques, il est équivalent de démontrer que $\Gamma$ ne satisfait
à aucune équation différentielle algébrique à coefficients dans $\C(x)$ ou dans $\C$.
En effet, sans trop formaliser les définitions (qui sont très intuitives et pour lesquels
on peut se reporter à \cite{RittDifferentialAlgebra}, \cite{Kolchin:differentialalgebraandalgebraicgroups} ou à \cite[\S2]{MarkusZetaGamma}, pour un résumé rapide),
le corps $\C(x)$ est différentiellement algébrique sur $\C$, car $\frac{d}{dx}(x)\in\C$.
On peut aussi se limiter à démontrer la transcendance différentielle de $\Gamma$ sur la
\emph{clôture différentielle} de $\C(x)$.
Cette dernière est une extension différentielle de $(\C(x),\frac{d}{dx})$
contenant une solution de tout système d'équations différentielles
algébriques à coefficients dans
$(\C(x),\frac{d}{dx})$, qui a une solution dans une extension différentielle quelconque de
$(\C(x),\frac{d}{dx})$. C'est bien l'analogue différentiel de la clôture algébrique.
\par
Considérons la fonction $\zeta$ de Riemann, \ie
la fonction obtenue par prolongement analytique de
$$
\zeta(x)=\sum_{n=1}^\infty \frac{1}{n^x},
\hbox{~pour tout $x\in\C$, $\Re(x)>0$.}
$$
Elle satisfait à l'équation fonctionnelle
$$
\zeta(x)=2(2\pi)^{x-1}\Gamma(1-x)\sin\l(\frac{\pi x}{2}\r)\zeta(1-x).
$$
Comme le terme $(2\pi)^{x-1}\sin\l(\frac{\pi x}{2}\r)$ est différentiellement algébrique, le théorème de Hölder implique immédiatement
la transcendance différentielle de $\zeta$, car si $\zeta$ était différentiellement algébrique la fonction
$\Gamma(x)=\zeta(1-x)\frac{1}{2}(2\pi)^{1-x}\l(\sin\frac{\pi(1-x)}{2}\r)^{-1}\zeta(x)^{-1}$ devrait l'être aussi:

\begin{coro}
La fonction $\zeta$ de Riemann est différentiellement transcendante sur $\C(x)$.
\end{coro}

En 1920 Ostrowski \cite{OstrowskiZeta} prouve aussi la transcendance différentielle sur $(\C(x),\frac{d}{dz},\frac{d}{dx})$ de
la fonction obtenue par prolongement analytique de la série
$$
\zeta(z,x)=\sum_{n=1}^\infty \frac{z^n}{n^x},
$$
en répondant à une question posée par Hilbert. Les résultats sur la fonction zeta de Riemann ont été
généralisés aussi dans plusieurs directions, souvent à l'aide de la théorie de Nevanlinna \cite{Laine}. Par contre, la question de l'indépendance
différentielle de $\Gamma$ et $\zeta$, c'est-à-dire de la propriété de $\Gamma$ et de $\zeta$ de
ne pas être solutions d'une équation différentielle algébrique en deux fonctions inconnues à coefficients dans $\C(x)$, est ouverte.
Pour les résultats sur la fonction $\zeta$ de Riemann, on
renverra plutôt aux travaux de B.Q. Li et Z. Ye, qui fournissent un survol  de la littérature sur le sujet
(voir \cite{LiYeActa}, \cite{LiYeNevalinna}, \cite{LiYeSurvey}).
On peut démontrer alors le corollaire suivant (qui généralise et simplifie le Théorème 3 dans \cite{MarkusZetaGamma}):

\begin{coro}\label{coro:markuszetagamma}
Soient $\Psi$ et $\Omega$ deux fonctions méromorphes sur $\C$ qui vérifient respectivement les équations fonctionnelles
$$
\Psi(x+1)=\Psi(x)
\hbox{~et~}
\Omega(x+1)=x\Omega(x).
$$
Si $\Psi(x)$ est différentiellemet transcendante sur $\C$ (ou, de façon équivalente sur $\C(x)$),
$\Psi(x)$ et $\Omega(x)$ sont différentiellement indépendantes sur $\C$  (ou sur $\C(x)$).
\end{coro}

\begin{proof}
Il existe une fonction $1$-périodique $\Pi(x)$ telle que
$\Omega(x)=\Pi(x)\Gamma(x)$.
Soit $\cC=\C\langle \Psi(x),\Pi(x)\rangle_{\frac{d}{dx}}\subset\cM er(\C)$ le corps
différentiel engendré par $\C$, $\Pi(x)$ et $\Psi(x)$ dans $\cM er(\C)$.
Le corps $\cF=\cC(x)$ vérifie les hypothèses du théorème précédent.
De plus, le sous-corps des éléments $1$-périodiques de $\cF$ coïncide avec $\cC$.
\par
Si on démontre que $\Omega$ est différentiellement transcendante sur $\cC(x)$ on pourra conclure que
$\Psi$ et $\Omega$ sont différentiellement indépendantes
sur $\C$. Si la fonction méromorphe $\Omega$ vérifiait une équation différentielle
algébrique à coefficients dans $\cF$, il en serait de même pour $\Gamma$ et, donc,
il existerait $g,f_0,f_1,\dots,f_n\in \cF$, avec $f_0,f_1,\dots,f_n$ périodiques de période $1$, telles que
$$
\sum_{i=0}^n f_i(x)\frac{d^i}{dx^i}\l(\frac{1}{x}\r)=
\frac{f_0(x)}{x}+\sum_{i=1}^n \frac{(-1)^{i}\,(i-1)!f_i(x)}{x^{i+1}}=g(x)-g(x+1).
$$
On remarque que $x$ est nécessairement transcendante sur $\cC$, car, si
le polynôme $P(T)\in\cC[T]$ s'annulait en $x$, il devrait s'annuler sur l'ensemble infini $x+\Z$.
On en déduit que la formule ci-dessus fournit une décomposition en éléments simples de
$g(x)-g(x+1)$ dans le corps des fonctions rationnelles $\cC(x)$.
Ceci est impossible car, si $g(x)-g(x+1)$ a un pôle en $x=0$, il doit aussi avoir au moins un autre pôle en
quelque $x\in\Z\smallsetminus\{0\}$.
\end{proof}

\begin{coro}
Les fonctions méromorphes $x\mapsto \zeta(\sin(2\pi x))$ et $\Gamma$
(resp. $x\mapsto \Gamma(\sin(2\pi x))$ et $\Gamma$) sont différentiellement indépendantes sur $\C(x)$.
\end{coro}

\begin{proof}
On démontre seulement les cas de $\zeta(\sin(2\pi x))$ et $\Gamma$.
Pour pouvoir appliquer le Corollaire \ref{coro:markuszetagamma}, il suffit de démontrer que  $\zeta(\sin(2\pi x))$ est différentiellement transcendante.
On sait que $\zeta$ est différentiellement transcendante sur $\C(x)$, c'est-à-dire que
la famille de fonctions $\l\{\frac{d^i\zeta}{dx^i}(x)\r\}_{i\geq 0}$ est algébriquement indépendante sur $\C(x)$.
Il s'ensuit que la famille $\l\{\frac{d^i\zeta}{dx^i}(\sin(2\pi x))\r\}_{i\geq 0}$ est algébriquement indépendante sur $\C(\sin(2\pi x))$ et donc sur son extension algébrique $\C(\sin(2\pi x),\cos(2\pi x))$.
Donc $\zeta(\sin(2\pi x))$ est différentiellement transcendante sur $\C(\sin(2\pi x),\cos(2\pi x))$ et
donc sur $\C$, car $\C(\sin(2\pi x),\cos(2\pi x))$ est une extension différentiellement algébrique de $\C$.
\end{proof}

\section{Théorie de Galois paramétrée}

La théorie de Galois paramétrée des équations différentielles et aux différences
est étudiée dans \cite{cassisinger} et \cite{HardouinSinger}.
Le cadre plus général est celui décrit dans ce dernier papier. Les auteurs considèrent
un corps $F$ équipé de deux familles finies de dérivations, $\Delta$ et $\Pi$, et d'une famille finie
d'automorphismes $\Sg$ et ils supposent que les éléments de $\Delta\cup\Pi\cup\Sg$ commutent deux à deux,
en tant qu'opérateurs agissant sur $F$.
Ils se donnent un système \emph{intégrable} d'équations matricielles\footnote{Le fait que le système
est intégrable signifie que les matrices $A_\sg$ et $B_\partial$
satisfont à des équations fonctionnelles liées à la commutativité des opérateurs; \cf Proposition
\ref{prop:integrabilite}.}
\beq\label{eq:syshs}
\l\{
\begin{array}{l}
\sg Y=A_\sg Y\hbox{~pour tout $\sg\in\Sg$}\\
\partial Y=B_\partial Y\hbox{~pour tout $\partial\in\De$}
\end{array}
\r.
\eeq
avec $A_\sg$ et $B_\partial$ matrice carrées à coefficients dans $F$, et $A_\sg$ inversible pour
tout $\sg\in\Sg$.
Moralement, il faut considérer $\Pi$ comme l'ensemble des dérivations associées à des paramètres du système.
À partir de cela,
ils construisent un groupe qui donne des informations sur les relations différentielles vérifiées par les solutions de \eqref{eq:syshs}
par rapport aux paramètres.
Dans le but de simplifier les notations de l'exposition qui suit, sans que cela simplifie vraiment les preuves,
on se placera dans un cadre moins général.

\subsection{Théorie de Picard-Vessiot paramétrée}
\label{subsec:DPV}

Considérons un corps différentiel aux différences, \ie un triplet
$(F,\sg,\partial)$, où $F$ est un corps, $\sg$ un automorphisme de $F$ et $\partial$ une dérivation de $F$, telle que $\partial\sg=\sg\partial$.
On suppose que $\sg$ n'est pas un automorphisme cyclique, bien que cette hypothèse ne soit nécessaire qu'à quelques endroits.
On dira que $F$ est un $(\sg,\partial)$-corps (et on utilisera sans les définir les concepts, très intuitifs, de
$(\sg,\partial)$-anneau, $(\sg,\partial)$-algèbre, ...; voir \cite{Levin:difference} et \cite{Cohn:difference} pour une exposition systématique de la théorie).
\par
La donnée initiale est celle d'un système aux différences
\beq\label{eq:sys}
\sg(Y)=AY,
\eeq
où $A\in GL_\nu(F)$ est une matrice inversible à coefficients dans $F$.

\begin{exem}
Typiquement on peut considérer le corps $\C(x)$ des fonctions rationnelles à
coefficients complexes avec les opérateurs suivants:
\begin{itemize}
\item $\tau: f(x)\mapsto f(x+1)$ et $\partial=\frac{d}{dx}$;
\item $\sgq: f(x)\mapsto f(qx)$, pour un $q\in\C$, $q\neq 0$ fixé, et $\partial=x\frac{d}{dx}$.
\end{itemize}
\end{exem}

\begin{defi}[Définition 6.10 dans \cite{HardouinSinger}]\label{defn:PV}
On appelle \emph{$(\sg,\partial)$-extension de Picard-Vessiot pour \eqref{eq:sys}}
un $(\sg,\partial)$-anneau $\cR$, extension de $F$, muni d'une extension de $\sg$ et $\partial$, préservant
la commutativité, \ie $[\sg,\partial]=0$, tel que:
\begin{enumerate}
\item
$\cR$ est un $(\sg,\partial)$-anneau simple,
\ie il n'a pas d'idéaux propres invariants par $\sg$ et $\partial$;

\item
$\cR$ est engendré, en tant que $\partial$-anneau, par une matrice inversible $Z \in GL_\nu(\cR)$ et $\frac{1}{det(Z)}$,
avec $Z$ solution de \eqref{eq:sys}.
\end{enumerate}
\end{defi}

Il est possible de construire formellement un tel objet.
Considérons l'anneau de $\partial$-polynômes
$$
F\{X,\det X^{-1}\}_{\partial}:=F\l[X_{i,j}^{(k)};\,i,j=1,\dots,n;\, k\geq 1\r]\l[\frac{1}{\det(X^{(1)}_{i,j})}\r],
$$
où $X_{i,j}^{(k)}$ sont des variables algébriquement indépendantes, telles que $\partial(X_{i,j}^{(k)})=X_{i,j}^{(k+1)}$.
Soient $X=(X_{i,j}^{(1)})$ et $X^{(k)}=\partial^k X$.
On définit sur $F\{X,\det X^{-1}\}_{\partial}$ une structure de $(\sg,\partial)$-algèbre, en posant
$\sg(X)=A X$ et
\beq\label{eq:sigmastructure}
\begin{array}{rcl}
\sg( X^{(k)})
&=&\sg(\partial^kX)=\partial^k(\sg(X))=\partial^k(AX)\\
&=&\ds\sum_{h=0}^k{k\choose h}\partial^{h}(A) X^{(k-h)},\hbox{~pour tout $k\geq 1$.}
\end{array}
\eeq
Le quotient $\cR$ de $F\{X,\det X^{-1}\}_{\partial}$ par un idéal invariant par $\sg$ et $\partial$ et maximal par cette propriété
(donc par un \emph{$(\sg,\partial)$-idéal maximal}) est bien sûr une $(\sg,\partial)$-extension de Picard-Vessiot pour \eqref{eq:sys}.
\par
Soit $K=F^\sg$ le sous-corps de $F$ des éléments invariants par $\sg$. La commutativité de $\sg$ et $\partial$
implique que $K$ est un corps différentiel par rapport à $\partial$.

\begin{prop}[Propositions 6.14 et 6.16 dans \cite{HardouinSinger}]
Si $(K,\partial)$ est différentiellement clos alors:
\begin{enumerate}
\item
Le sous-anneau des constantes $\cR^\sg$
d'une $(\sg,\partial)$-extension de Picard-Vessiot $\cR$ pour \eqref{eq:sys}
coïncide avec $K$, c'est-à-dire que $\cR$ ne contient pas de nouvelles constantes par rapport à $F$.

\item
Deux $(\sg,\partial)$-extensions de Picard-Vessiot pour \eqref{eq:sys} sont isomorphes en tant que $(\sg,\partial)$-anneaux.
\end{enumerate}
\end{prop}

\begin{rema}\label{rema:descente}
Si  $K$ est seulement algébriquement clos et la $(\sg,\partial)$-extension de Picard-Vessiot $\cR$ est en plus un $\sg$-anneau simple,
alors le point 1 de la proposition ci-dessus est encore vrai.
Par contre il faut en général procéder à une extension des constantes pour
avoir un isomorphisme entre deux $(\sg,\partial)$-extensions de Picard-Vessiot.
M. Wibmer a affiné la construction donnée ci-dessus pour obtenir
une $(\sg,\partial)$-extension de Picard-Vessiot qui est aussi un $\sg$-anneau simple,
\cf \cite{Wibmchev} et \cite{wibmer2011existence} (son argument est aussi repris dans \cite{diviziohardouinPacific}). Pour cela il construit de façon fine
un $(\sg,\partial)$-idéal maximal de $F\{X,\det X^{-1}\}_{\partial}$, qui est aussi un $\sg$-idéal maximal, en partant d'un $\sg$-idéal maximal de $F[X,\det X^{-1}]$
qu'il prolonge en le dérivant.
Ces questions de descente sont traitées en toute généralité, par des méthodes tannakiennes,
dans \cite{GilletGorchinskyOvchinnikov}.
\end{rema}

\subsection{Groupe de Picard Vessiot paramétré}
\label{subsec:difPVgr}
Supposons, pour simplifier, que le corps des $\sg$-constantes $(K,\partial)$ est différentiellement clos.

\medskip
Soit $\cR$ une $(\sg,\partial)$-extension de Picard-Vessiot
pour \eqref{eq:sys}.
Comme dans la théorie de Galois des équations aux différences non paramétrées (voir \cite{vdPutSingerDifference}),
$\cR$ n'est pas, en général, un anneau intègre, mais il est la somme directe de copies d'un anneau intègre, de façon
qu'on peut considérer son corps total des fractions $L$, qui est isomorphe à une somme directe de copies d'un même corps
(\cf \cite{HardouinSinger}).

\begin{defi}\label{defn:pvgr}
Le groupe $Gal^{\partial} (A)$ (qu'on note aussi $Aut^{\sg,\partial}(L / F)$)
des automorphismes de  $L$, qui fixent $F$ et commutent avec $\sg$ et $\partial$, est le
\emph{groupe de Galois paramétré} de \eqref{eq:sys}. On l'appellera aussi \emph{$\partial$-groupe de Galois} de \eqref{eq:sys}.
\end{defi}

\begin{rema}
Le groupe $Gal^{\partial}(A)$ agit sur une matrice fondamentale $Z\in GL_\nu(L)$ de solutions
de \eqref{eq:sys}.
Pour tout $\varphi\in Gal^{\partial}(A)$, la matrice $\varphi(Z)$ est encore une solution de \eqref{eq:sys}, donc il existe
$U\in GL_\nu(K)$ telle que $\varphi(Z)=ZU$, avec $\sg(ZU)=\sg(Z)U=AZU$.
Cette action fournit une représentation fidèle de $Gal^{\partial}(A)$ dans $GL_\nu(K)$, dont
l'image est formée des $K$-points d'un $\partial$-groupe algébrique linéaire de
$GL_\nu(K)$, dans le sens de Kolchin.
C'est-à-dire que c'est un sous-groupe de
$GL_\nu(K)$ définit par un $\partial$-idéal de $K\l\{X,\det X^{-1}\r\}_{\partial}$, donc un lieu de zéros d'un ensemble fini d'équations différentielles
à coefficients dans $K$.
Comme $(K,\partial)$ est un corps différentiellement clos, nous pouvons nous contenter ici d'une description naïve de ce groupe, via son ensemble de points $K$-rationnels.
On aura tendance à ne pas faire très attention à distinguer les groupes de Galois et leur représentations en tant que sous-groupes de $GL_\nu$.
\end{rema}

On reconnaîtra dans la proposition ci-dessous le c{\oe}ur
de la correspondance de Galois, qu'on n'énoncera pas en entier.
On n'aura pas de difficulté à en imaginer les énoncés
en s'inspirant de la théorie de Galois classique.

\begin{prop}[Lemme 6.19 dans \cite{HardouinSinger}]~
\begin{enumerate}
\item L'anneau $L^{Gal^{\partial} (A)}$ des éléments de $L$ fixés par $Gal^{\partial} (A)$ coïncide avec $F$.
\item Soit $H$ un $\partial$-sous-groupe algébrique de $Gal^{\partial}(A)$.
Si $L^H = F$, alors $H = Gal^{\partial} (A)$.
\end{enumerate}
\end{prop}

Le groupe de Galois (non paramétré) $Gal(A)$ de
\eqref{eq:sys} sur K est construit de la façon suivante: on considère le quotient de l'algèbre de polynômes
$F[X,\det X^{-1}]$, munie de l'action de $\sg$ définie par $\sg(X)=A(X)$, par un $\sg$-idéal maximal, et son corps total des franctions $L$; alors
$Gal(A)$ est le groupe d'automorphismes de $L/F$ qui commutent avec $\sg$
(voir \cite{vdPutSingerDifference}). Nous avons:

\begin{prop}[Proposition 6.21 in \cite{HardouinSinger}]
Le groupe algébrique $Gal(A)$ est la clôture de Zariski de $Gal^{\partial} (A)$ (dans $GL_\nu(K)$).
\end{prop}

\begin{rema}
Si $F$ a un corps des constantes $K$ algébriquement clos et si on considère une $(\sg,\partial)$-extension de Picard-Vessiot de $F$,
en suivant la construction de \cite{wibmer2011existence}, on peut construire un schéma en $\partial$-groupes défini sur $K$, dont les points
rationnels sur la clôture différentielle de $K$ peuvent être identifiés avec $Gal^\partial(A)$.
Évidemment, pour définir $Gal^\partial(A)$ il faut considérer la clôture différentielle $\wtilde K$ de $K$
et travailler sur le corps des fractions de $F\otimes_K \wtilde K$, avec $\sg$ agissant sur $\wtilde K$ comme l'identité ($F$ et $\wtilde K$ étant
linéairement disjoints).  Pour plus de détails voir \cite[\S1.2]{diviziohardouinPacific}.
\end{rema}

\subsection{Dépendance différentielle}
\label{sec:difdependency}

Une $(\sg,\partial)$-extension de Picard-Vessiot $\cR$ de $F$
pour \eqref{eq:sys} est un $Gal^{\partial} (A)$-torseur, dans le sens de Kolchin.
Cela implique, en particulier, que toutes les relations différentielles par rapport à
la dérivation $\partial$, satisfaites par une matrice fondamental de solutions de
\eqref{eq:sys}, sont entièrement déterminées par le groupe
$Gal^\partial(A)$:

\begin{theo} [Proposition 6.29 dans \cite{HardouinSinger}]\label{prop:degtrs}
Le degré de $\partial$-transcendance de $\cR$ sur $F$ est égal  à la
$\partial$-dimension de $Gal^\partial(A)$.
\end{theo}

Les notions de $\partial$-transcendance et $\partial$-dimension sont celles intuitives, notamment
le degré de $\partial$-transcendance de $\cR/F$ est égal au nombre maximal d'élément différentiellement
indépendants de $\cR$ sur $F$ et la
$\partial$-dimension de $Gal(A)$ est égal au degré de $\partial$-transcendance de son algèbre de
Hopf différentielle sur le corps des constantes $K$.
En gros, ce résultat dit que plus le groupe est petit, plus il y a des relations différentielles entre
les solutions de \eqref{eq:sys} dans $\cR$.

\subsection{Équations aux différences linéaires d'ordre $1$}

Il n'est pas difficile de se convaincre que les
sous-groupes différentiels de ${\mathbb G}_a^n$ sont définis par des équations différentielles
linéaires (voir \cite{cassdiffgr}). On déduit du Théorème \ref{prop:degtrs} le critère:

\begin{prop}[{Proposition 3.1 dans \cite{HardouinSinger}}]\label{prop:gatrans}
Soient  $a_1, ..., a_n$ des éléments non nuls de $F$
et $S$ une $(\sg,\partial)$-extension de $F$ telle que $S^\sg=F^\sg =K$.
Si $z_1,...,z_n \in S$ sont solutions des équations aux différences
$\sg (z_i) -z_i =a_i$, pour $i=1,...,n$,
alors $z_1,...,z_n \in S$ satisfont à une $\partial$-relation
différentielle non banale sur $F$
si et seulement s'il existe un polynôme différentiel linéaire homogène non nul
$L(Y_1,...,Y_n)$ à coefficients dans $K$ et un élément $f \in F$ tels que
$L(a_1, ..., a_n) =\sg (f)-f$.
\end{prop}

On remarquera la similitude entre cet énoncé et le Théorème \ref{theo:BankKaukmannGamma}.
En effet, si on considère la dérivée logarithmique de l'équation de la fonction Gamma
$$
z(x+1)=z(x)+\frac{1}{x},
$$
on en déduit facilement un énoncé analogue sur le corps $F$, ayant un corps des constantes $K$
différentiellement clos par rapport à $\partial$.
Cette dernière hypothèse n'est pas vérifiée dans le cas des fonctions méromorphes.
Il est néanmoins possible de prouver un critère de ce type pour les solutions méromorphes.
On reviendra de nouveau sur ce point.

\subsection{Intégrabilité}

La proposition suivante établit le lien entre la structure du
$\partial$-groupe de Galois et l'intégrabilité du système aux différences
par rapport à l'opérateur différentiel.
Ce genre de problématique se retrouve très naturellement lorsque, par exemple, on
cherche une paire de Lax pour une équation qui mélange opérateur différentiels et aux différences.
Ce type d'équations est appelé \emph{équations à retard} (\emph{delay equations} dans la littérature en anglais), ou bien, dans le cas
spécifique des équations aux $q$-différences, équations du pantographe.
Elles se retrouvent naturellement
lorsque l'équation décrit un système dépendant de la variable libre,
disons le temps $t$, à la fois de façon continue et discrète.
La définition de $\partial$-groupe constant est expliquée immédiatement
après l'énoncé.

\begin{prop}[{Proposition 2.9 dans \cite{HardouinSinger}}]\label{prop:integrabilite}
Les assertions suivantes sont équivalentes:
\begin{enumerate}

\item
Le $\partial$-groupe de Galois
$Gal^\partial (A)$ est conjugué sur $K$ avec un $\partial$-groupe constant.

\item
Il existe $B\in M_n(F)$ tel que le système
$$
\left\{\begin{array}{l} \sg(Y)=AY \\
\partial Y=BY
\end{array} \right.
$$
est intégrable, c'est-à-dire que les matrices $B$ et $A$ satisfont à l'équation fonctionnelle
suivante, induite par la commutativité entre $\sg$ et $\partial$:
$$
\sg(B)A= \partial(A) +A B.
$$
\end{enumerate}
\end{prop}

Soient $K$ un $\partial$-corps et $C$ son sous-corps des $\partial$-constantes.
On dit qu'un $\partial$-groupe linéaire $G \subset GL_\nu$ défini sur $K$
est un \emph{$\partial$-groupe constant} (ou, plus brièvement, qu'il est $\partial$-constant)
si son idéal de définition dans
$K\{X,\frac{1}{\det\left(X\right)}\}_\partial$
contient les polynômes différentiels
$\partial(X_{i,j})$, pour tout $i,j=1,\dots,\nu$.
Puisque $K$ est différentiellement clos, cela est équivalent
au fait que les points $K$-rationnels de $G$ coïncident avec
les points $C$-rationnels d'un groupe linéaire défini sur $C$.
On en déduit le corollaire suivant:

\begin{coro}
Considérons un système \eqref{eq:sys} à coefficients dans $F$ et son $\partial$-groupe de Galois
$Gal^\partial(A)$. S'il existe une représentation fidèle $\varrho:Gal^\partial(A)\hookrightarrow GL_\mu(K)$
et une matrice dans l'image de $\varrho$ dont le polynôme minimal n'est pas à coefficients dans
$C$, alors \eqref{eq:sys} n'est pas intégrable au sens de la proposition précédente.
\end{coro}

\begin{rema}
Il existe un critère d'intégrabilité analogue pour des équations aux différences dépendant de plusieurs paramètres.
Dans le cas des équations différentielles d'ordre 2, dépendant de plusieurs paramètres, il est possible
de vérifier l'intégrabilité paramètre par paramètre pour conclure à l'intégrabilité globale \cite{dreyfus2011kovacic}.
Ce résultat a été prouvé aussi pour les équations différentielles d'ordre quelconque dans \cite{gorchinskiy2012isomonodromic}.
La preuve repose sur des théorèmes de structure des groupes algébriques différentiels et donc un résultat analogue devrait être
vrai aussi pour les équations aux différences.
\end{rema}

Selon \cite{cassdiffgr}, si $H$ un $\partial$-groupe sur un corps différentiellement clos $K$,
dont la clôture de Zariski est
un groupe algébrique linéaire simple $G$ sur $K$,
alors soit $H=G$ soit $H$ est conjugué sur $K$ à un $\partial$-groupe constant.
Compte tenu de la Proposition \ref{prop:degtrs}, on obtient:

\begin{coro}
Si $Gal(A)$ est un groupe algébrique simple,
soit nous sommes dans la situation de la Proposition \ref{prop:integrabilite}
soit il n'existe aucune relation différentielle non banale entre les éléments d'une matrice fondamentale
de solutions de
$\sg(Y)=AY$ à coefficients dans $L$.
\end{coro}

\begin{rema}
Si des relations algébriques entre les éléments
d'une matrice fondamentale de solutions existent,
on peut toujours en déduire des relations différentielles par dérivation. On peut
considérer
que celles-ci sont des relations différentielles banales.
\end{rema}

\section{Transcendance différentielle des solutions méromorphes}

On a vu que la théorie de Galois fournit des critères
de transcendance différentielle pour des solutions abstraites d'une équation aux différences.
Dans le cas de la fonction Gamma d'Euler, par exemple, en s'inspirant de
la construction plus haut, nous pourrions considérer
l'anneau
$$
\cR_\Gamma=\mathcal P(x)\l[\Gamma(x), \Gamma^\p(x),\Gamma^{(2)}(x),\dots,\frac{1}{\Gamma(x)}\r],
$$
où $\cP$ est le corps des fonctions méromorphes sur $\C$ et $1$-périodiques.
L'anneau $\cR$ est bien un $(\tau,\partial)$-anneau, par rapport à l'opérateur
$\tau:f(x)\mapsto f(x+1)$ et à la dérivation  $\partial=\frac{d}{dx}$, et,
puisque $\cR_\Gamma\subset\cM er(\C)$,
ses constantes coïncident avec $\cP$. Par contre il est assez difficile, en général, d'établir si
$\cR_\Gamma$ est un $(\tau,\partial)$-anneau simple, donc toute la discussion précédente tombe (ou risque de tomber)
à l'eau.
Pour s'en sortir, il suffit de considérer que
la clôture différentielle $\wtilde\cP$ de $\cP$ par rapport à $\partial$.
Il n'est pas difficile de voir que le corps $\cM er(\C)$ des fonctions méromorphes sur $\C$ et $\wtilde\cP$
sont linéairement disjoints sur $\cP$ (voir le Lemme \ref{lemm:lindisj} ci-dessus).
On peut alors considérer l'anneau $\cR_\Gamma\otimes_\cP\wtilde\cP$
et le comparer à une $(\tau,\partial)$-extension de Picard-Vessiot, au sens de la Définition \ref{defn:PV}.
On va formaliser ces considérations.
\par
On appellera $(\cF,\sg,\partial)$ l'un des deux $(\sg,\partial)$-corps\footnote{En réalité,
nous n'avons pas besoin de fixer un choix pour $\partial$: les propositions qui suivent
sont vraies pour toute dérivation commutant avec
les deux choix de $\sg$ ci-dessous. Ca sera le cas dans \S\ref{sec:qdiff}.}
suivants:
\begin{enumerate}
\item
Le corps $(\cF,\sg,\partial)$ est une extension de $(\cP(x),\tau,\frac{d}{dx})$
contenue dans $(\cM,\sg,\partial):=(\cM er(\C),\tau,\frac{d}{dx})$.

\item
Pour $q\in\C$, $|q|\neq 1$, on considère le corps des fonctions elliptiques
$\cE_q$, autrement dit le sous-corps du corps $\cM er (\C^*)$ des fonctions méromorphes
sur $\C^*$ des fonctions invariantes par $\sg:=\sgq:f(qx)\mapsto f(x)$. Dans
ce cas on considère une extension $(\cF,\sg,\partial)$ de
$(\cE_q(x),\sgq,x\frac{d}{dx})$ contenue dans $\cM=\cM er(\C^*)$.
\end{enumerate}

Ces deux situations ont beaucoup en commun, mais diffèrent par la nature différentielle du corps des constantes.
En effet, le corps $\cK$ des éléments $\sg$-invariants de $\cF$ coïncide avec celui de $\cM$, donc $\cK=\cP$ pour $\sg=\tau$
et $\cK=\cE_q$ si $\sg=\sgq$.
Il est bien connu que le corps des fonctions elliptiques $\cE_q$ est différentiellement algébrique.
Pour le voir il est suffisant de passer de la notation multiplicative à la notation additive et de
se souvenir du fait que la fonction $\wp(x)$ de Weierstrass satisfait à une équation différentielle
d'ordre $2$.
D'un autre côté, on a vu que $\cP$ contient au moins $x\mapsto\zeta(\sin(2\pi x))$, qui est différentiellement
transcendant.
Néanmoins on a:

\begin{lemm}\label{lemm:lindisj}
La clôture différentielle $\wtilde\cK$ de $\cK$ et le corps $\cM$ (resp. $\cF$)
sont linéairement disjoints sur $\cK$.
\end{lemm}

\begin{proof}
Soit $\{\a_i\}_{i\in I}$ une famille finie d'éléments de $\wtilde\cK$ linéairement indépendants
sur $\cK$, mais qui deviennent liés sur $\cM$ (resp. $\cF$) en tant qu'éléments de $\cF\otimes_\cK\wtilde\cK$.
On suppose qu'elle est minimale, c'est-à-dire que
pour tout $\iota\in I$ la famille $\{\a_i\}_{i\in I,i\neq\iota}$ reste linéairement
indépendante sur $\cM$ (resp. $\cF$).
Soit $\sum_i\la_i\a_i=0$ une combinaison linéaire non banale des $\a_i$ sur $\cM$ (resp. $\cF$). On peut supposer qu'il
existe $\iota\in I$ tel que $\la_\iota=1$. On obtient une contradiction en
considérant $\sum_i(\la_i-\sg(\la_i))\a_i=0$.
\end{proof}

Soit $\sg Y=AY$ un système aux différences tel que $A(x)\in GL_\nu(\cF)$,
ayant une matrice fondamentale de solutions $U\in GL_\nu(\cM)$.
On appelle $\cR_{\cM}$ l'anneau $\cF\{U, \det U^{-1}\}_\partial\subset\cM$ et $\cR^\p_{\cM}$ un quotient de
l'anneau des polynômes différentiels $\cF\{X, \det X^{-1}\}_\partial$ par un $(\sg,\partial)$-idéal maximal. On note aussi $\cR$ la $(\sg,\partial)$-extension
de Picard-Vessiot sur $\wtilde\cF=Frac(\cF\otimes_\cK\wtilde\cK)$ associée à $\sg Y=AY$.

\begin{lemm}
$\cR_\cM\otimes_\cK\wtilde\cK
\cong\cR_\cM^\p\otimes_\cK\wtilde\cK
\cong\cR$.
\end{lemm}

\begin{proof}
Voir le Corollaire 3.3 et la Proposition 3.4 dans \cite{diviziohardouinComp} (et \cite{ChatHardouinSinger}
pour le cas non paramétré), dans le cas
$(\cF,\sg,\partial)=(\cE_q(x),\sgq,x\frac{d}{dx})$. La preuve se généralise sans difficulté.
\end{proof}

Moralement, la proposition précédente dit que les groupes
$Aut^{\sg,\partial}(\cR_\cM/\cF)$,
$Aut^{\sg,\partial}(\cR_\cM^\p/\cF)$ et
$Aut^{\sg,\partial}(\cR/\wtilde\cF)$ coïncident. Cette affirmation n'a pas vraiment
de sens car les deux premiers groupes peuvent ne pas avoir beaucoup d'éléments, à cause du
fait que $\cK$ n'est pas différentiellement clos. Il est par contre possible de donner un sens
rigoureux à cette affirmation en utilisant les schémas en groupes et les catégories tannakiennes différentielles,
introduite dans \cite{ovchinnikovdifftannakian} (voir aussi \cite{kamensky2009model}, \cite{kamensky2011tannakian} and \cite{GilletGorchinskyOvchinnikov}).
En effet, chacun de ces anneaux détermine un foncteur fibre pour la catégorie tannakienne différentielle
engendrée par le module aux différences associé à $\sg Y=AY$. Les schémas en groupes des automorphismes
tensoriels de ces foncteurs deviennent tous isomorphes deux à deux sur $\wtilde\cK$.
On en déduit:

\begin{theo}
Il existe un $\partial$-groupe algébrique $G_\cK$ défini sur $\cK$ tel que
$Aut^{\sg,\partial}(\cR_\cM/\cF)$ est le groupe de $\cK$-points de
$G_\cK$ et que $G_\cK\otimes\wtilde\cK\cong Gal^\partial(A)$.
\end{theo}

Ceci nous permet de donner une preuve d'un analogue de la Proposition
\ref{prop:gatrans} sur un corps de fonctions méromorphes, qui est cachée entre la Proposition 3.1 et
le Corollaire 3.2 de \cite{HardouinSinger} (voir aussi \cite{hardouincompositio}).
Une fois de plus, on utilise de façon cruciale la classification des sous-groupes
différentiels de $\mathbb G_a^n$ dans \cite{cassdiffgr}.

\begin{prop}\label{prop:gatransmero}
Soient  $a_1, ..., a_n$ des éléments non nuls de $\cF$.
Si $z_1,...,z_n \in\cM$ satisfont aux équations aux différences
$\sg (z_i) -z_i =a_i$, pour $i=1,...,n$,
alors $z_1,...,z_n$ satisfont à une $\partial$-relation
différentielle sur $\cF$
si et seulement s'il existe un polynôme différentiel linéaire homogène non nul
$L(Y_1,...,Y_n)$ à coefficients dans $\cK$ et un élément $f \in\cF$ tels que
$L(a_1, ..., a_n) =\sg (f)-f$.
\end{prop}

\begin{rema}
Pour $\sg=\tau$, on retrouve une preuve du Théorème \ref{theo:BankKaukmannGamma},
avec l'hypothèse supplémentaire que $\cP(x)\subset\cF$. On reviendra sur le problème
de descente de $\cP$ à $\C(x)$.
\par
On en déduit aussi immédiatement que toute solution méromorphe de l'équation
$\Omega(x+1)=x\Omega(x)$ est différentiellement transcendante sur
$\cP$, ce qui prouve le Corollaire \ref{coro:markuszetagamma}.
\end{rema}

\begin{proof}
Une implication a déjà été prouvée dans la Remarque \ref{rema:implicationfacile}.
Considérons l'anneau $\cR_\cM$ associé au système aux différences
$\sg Y=AY$, où $A$ est une matrice diagonale par blocs:
$$
A=diag\l(\begin{pmatrix}1&a_1\\0&1\end{pmatrix},\dots,
\begin{pmatrix}1&a_n\\0&1\end{pmatrix}\r).
$$
Une matrice fondamentale de solutions de $\sg Y=AY$ est donnée par:
$$
U=diag\l(\begin{pmatrix}1&z_1\\0&1\end{pmatrix},\dots,
\begin{pmatrix}1&z_n\\0&1\end{pmatrix}\r)\in  GL_{2n}(\cM).
$$
Il s'ensuit que $Gal^\partial(A)$ est un $\partial$-sous-groupe de $\mathbb G_a^n$ défini sur $\cK$.
Par hypothèse, c'est un sous-groupe propre (\cf Proposition \ref{prop:degtrs}).
Il existe donc un polynôme différentiel linéaire homogène non nul
$L(Y_1,...,Y_n)$ à coefficients dans $\cK$, contenu dans l'idéal de définition de
$Gal^\partial(A)$.
On pose $f=L(z_1,\dots,z_n)\in\cR_\cM$.
Un argument galoisien montre que $f$ est invariant par l'action de $Gal^\partial(A)$ et donc que $f\in\cF$.
On en déduit que
$$
0=\sg(L(z_1,\dots,z_n)-f)-(L(z_1,\dots,z_n)-f)
=L(a_1,\dots,a_n)-(\sg(f)-f).
$$
Pour plus de détails voir la Proposition 3.1 dans \cite{HardouinSinger}.
\end{proof}

\begin{coro}[Corollaire 3.2 dans \cite{HardouinSinger}; \cite{hardouincompositio}]
Soient  $a_1, ..., a_n$ des éléments non nuls de $\C(x)$ et
$z_1,...,z_n \in\cM$ des solutions méromorphes des équations aux différences
$\sg (z_i) -z_i =a_i$, pour $i=1,...,n$.
Les assertions suivantes sont équivalentes:
\begin{enumerate}
\item
Les fonctions $z_1,...,z_n$ satisfont à une $\partial$-relation
différentielle sur $\cK(x)$.
\item
Il existe un polynôme différentiel linéaire homogène non nul
$L(Y_1,...,Y_n)$ à coefficients dans $\cK$ et un élément $f \in \cK(x)$ tels que
$L(a_1, ..., a_n) =\sg (f)-f$.
\item
Les fonctions $z_1,...,z_n$ satisfont à une $\partial$-relation
différentielle sur $\C(x)$.
\item
Il existe un polynôme différentiel linéaire homogène non nul
$L(Y_1,...,Y_n)$ à coefficients dans $\C$ et un élément $f \in \C(x)$ tels que
$L(a_1, ..., a_n) =\sg (f)-f$.
\end{enumerate}
\end{coro}

\begin{proof}
La proposition précédente donne l'équivalence entre 1. et 2.
L'implication $4.\Rightarrow 3.$ se prouve comme la Remarque \ref{rema:implicationfacile} et
l'implication $3.\Rightarrow 1.$ est tautologique.
Il ne nous reste qu'a démontrer que $2.\Rightarrow 4.$
Pour cela on va utiliser un argument de descente classique.
On considère un polynôme différentiel linéaire homogène $\wtilde L$ et une fonction rationnelle
$\wtilde f$ en $x$, obtenus des $L$ et $f$ en remplaçant leurs coefficients (dans $\cK$) par des coefficients
génériques. L'identité $L(a_1, ..., a_n) =\sg (f)-f$ se traduit en une série
d'équations algébriques en les coefficients de $\wtilde L$ et $\wtilde f$,
à coefficients dans $\C$. Ces équations ont une solution dans $\cK$, car
$L$ et $f$ existent par hypothèse. On conclut qu'elles doivent avoir une solution dans $\C$, puisque $\C$ est algébriquement
clos. Ceci termine la preuve.
\end{proof}

\section{Le cas des équations aux $q$-différences}
\label{sec:qdiff}

Les résultats du paragraphe précédent s'appliquent aussi bien aux équations
aux différences finies qu'aux équations aux $q$-différences.
Considérons un nombre complexe $q$ tel que $|q|>1$ et la fonction Theta de
Jacobi
$$
\theta_q(x)=\sum_{n\in\Z}q^{-n(n-1)/2} x^n.
$$
Elle vérifie l'équation aux $q$-différences $y(qx)=qxy(x)$.
La dérivée logarithmique $\ell_q(x)$ de $\theta_q(x)$ par rapport à la dérivation
$\partial=x\frac{d}{dx}$ vérifie l'équation
$$
\ell_q(qx)=\ell_q(x)+1.
$$
Il s'ensuit que $\partial(\ell_q)\in\cE_q$ et que, sans surprise, la fonction Theta de Jacobi est différentiellement
algébrique.
Si on avait voulu appliquer la Proposition \ref{prop:gatransmero} à l'équation de $\ell_q(x)$
il aurait suffit de  poser $f=\partial(\ell_q)$ et $L=\partial$.

\begin{rema}
L'algébricité différentielle de $\Theta_q$ est équivalente
au fait que le $\partial$-groupe de Galois $Gal^\partial(qx)$ est un sous-groupe différentiel
propre de $\mathbb G_m$.
Dans le cas différentiel, $\mathbb G_m$ se plonge dans $\mathbb G_a$ grâce à la dérivée logarithmique
$z\mapsto \partial(z)/z$. On peut prouver que les sous-groupes différentiels propres non finis
de $\mathbb G_m$, définis sur $\cE_q$,
ont un idéal de définition engendré par un nombre fini d'équations différentielles $\cL(\partial(z)/z)=0$, où
$\cL$ est un opérateur différentiel linéaire dans $\cE_q[\partial]$.
Il n'est pas difficile de voir que $Gal^\partial(qx)\subset\l\{\partial\l(\frac{\partial z}{z}\r)=0\r\}\subset\mathbb G_m$.
\end{rema}

Une problématique propre aux équations aux $q$-différences est celle liée
à la dépendance différentielle en $q$ des solutions, lorsque
$q$ est un paramètre (voir \cite{diviziohardouinPacific}).
Par exemple, si on pose $\partial_q=q\frac{d}{dq}$ et $\partial_x=x\frac{d}{dx}$,
la fonction Theta de Jacobi vérifie l'équation aux dérivées partielles
$$
2\partial_q\theta_q=-\partial_x^2\theta_q+\partial_x\theta_q,
$$
laquelle est, à un changement de variable près, l'équation de la chaleur.
Il est possible de déduire des arguments ci-dessus qu'il n'y a guère que la
fonction $\theta_q$ qui vérifie une équation aux $q$-différences d'ordre $1$ à coefficients dans $\C(x)$
et qui satisfait à des relations différentielles par rapport à $\partial_q,\partial_x$.
\par
Commençons par formaliser le cadre.
On considère le corps $\C(q)$ avec la norme $q^{-1}$-adique, c'est-à-dire qu'on fixe un
réel $d>1$ et pour tout $f(q),g(q)\in\C[q]$, avec $g(q)\neq 0$ on pose:
$$
\l|\frac{f(q)}{g(q)}\r|=d^{\deg_q f-\deg_q g}.
$$
Ceci définie une norme ultramétrique sur $\C(q)$ qui s'étend à la plus petite extension normée
$\cC$ de $\C(q)$, complète et algébriquement close. On peut alors considérer
les fonctions méromorphes $\cM$ sur $\cC^*=\cC\smallsetminus\{0\}$, qui sont les quotients
de séries entières à coefficients dans $\cC$, ayant un rayon de convergence infini.
Les opérateurs $\sgq,\partial_q,\partial_x$ s'étendent naturellement à $\cM$ et
on peut considérer le corps $\cE_q$ des fonctions elliptiques, \ie $\sgq$-invariantes, de $\cM$.
\par
Nous allons considérer le corps des fonctions méromorphes
$\cF=\cE_q(x,\ell_q(x))\subset\cM$.
Puisque $\ell_q(qx)=\ell_q(x)+1$, le corps $\cF$ est stable par $\sgq$.
Évidemment le triplet $(\cF,\sgq,\partial_x)$ se comporte exactement comme les corps considérés
dans la section précédente, bien que la nature des fonction méromorphes dans ce contexte soit un peu différente.
Bien sûr, le corps des $\sgq$-invariants de $\cF$ coïncide avec $\cE_q$.
\par
Si on pose $\de=\ell_q(x)\partial_x+\partial_q$, on peut vérifier que $\de$ commute avec $\sgq$ (voir Lemme 2.1 dans \cite{diviziohardouinPacific}), que
$\de(\ell_q)\in\cE_q$ et que, donc, elle laisse $\cF$ stable dans $\cM$.
Il s'ensuit qu'aussi le triplet $(\cF,\sgq,\de)$ est de la même nature que les corps
différentiels/aux différences considérés précédemment. Son sous-corps des $\sgq$-invariants est toujours $\cE_q$.

\begin{rema}
On déduit de l'équation $\ell_q(qx)=\ell_q(x)+1$ que $\de\ell_q(x)\in\cE_q$, ce qui prouve que
$\theta_q(x)$ vérifie une équation différentielle non banale en $\de$.
Comme on l'a déjà remarqué, ceci est équivalent au fait que le $\de$-groupe de Galois $Gal^\de(qx)$ est un sous-groupe différentiel
propre de $\mathbb G_m$.
Le calcul du $\de$-groupe de Galois $Gal^\de(qx)$ est étroitement lié à l'équation de la chaleur
(voir (2-3) dans \cite{diviziohardouinPacific}).
\end{rema}

La Proposition \ref{prop:gatransmero} est valable pour $(\cF,\sgq,\partial_x)$ et pour
$(\cF,\sgq,\de)$, avec exactement la même preuve (voir le Corollaire 2.5 dans \cite{diviziohardouinPacific}).
Puisque $\cF$ est une extension purement transcendante de $\cE_q$, on peut en déduire, par
un argument élémentaire de décomposition en éléments simples, la proposition suivante:

\begin{prop}
Soient $a(x)\in\C(q,x)$ et $u\in\cM$ une solution de $y(qx)=a(x)y(x)$.
Les affirmations suivantes sont équivalentes:
\begin{enumerate}
\item
Il existe $r\in\Z$, $g(x)\in\C(q,x)$ et $\mu\in\C(q)$ tels que
$a(x)=\mu x^rg(qx)/g(x)$.
\item
La fonction $u$ est solution d'une équation différentielle algébrique
non triviale sur $(\cF,\partial_x)$ (et donc sur $\cC(x)$).
\item
La fonction $u$ est solution d'une équation différentielle algébrique
non triviale sur $(\cF,\de)$  (et donc sur $\cC(x)$).
\end{enumerate}
\end{prop}

\begin{rema}
L'équivalence entre la première et la deuxième assertion est le Théorème 1.1 dans \cite{HardouinSinger},
alors que l'équivalence entre la première et la troisième affirmation est prouvée dans
la Proposition 2.7 de \cite{diviziohardouinPacific}.
\par
Une solution méromorphe de $y(qx)=a(x)y(x)$, avec $a(x)=\mu x^rg(qx)/g(x)$,
est donnée par:
$$
\theta_q(\mu x/q^r)\theta_q(x)^{r-1}g(x)\in\cM.
$$
Il n'y a, donc, guère que la fonction Theta de Jacobi, qui soit solution d'une équation aux $q$-différences d'ordre $1$ et
qui ait des propriétés d'algébricité différentielle
non banales par rapport a $\partial_x,\partial_q$.
\par
Signalons le fait qu'on peut aussi étudier l'intégrabilité des systèmes aux $q$-différences d'ordre $>1$ par rapport à
$\partial_x$ et $\partial_q$ (voir le Corollaire 2.9 dans \cite{diviziohardouinPacific}).
\end{rema}

\section{Quelques mots sur ce que ce survol ne contient pas}

Ce survol est un introduction à des thématiques galoisiennes liées
aux équations aux différences et à la transcendance différentielle.
On a rapidement dû renoncer à la velléité de donner une liste de références relativement complète
sur le sujet de la transcendance différentielle, car la littérature est tentaculaire.
L'article de survol de Rubel \cite{Rubelsurvey}, ainsi que \cite{Rubelpbs1} et \cite{Rubelpbs2}, fournissent
une jolie vue panoramique des travaux plus classiques. On renvoie le lecteur à ces articles et à leur bibliographie.
On signale aussi:
\begin{itemize}
\item
Dans [a], [b], [c] on trouvera un approche effectif \`a la transcendance diff\'erentielle, dans un style
diophantien.
\item
Dans \cite{MarkusZetaGamma} on trouve une allusion aux liens entre transcendance différentielle
et dynamique holomorphe. Sur ce point la littérature semble se limiter aux articles \cite{Bergweilerpbrubel,BeckerBergweiler}
\item
En combinatoire, il arrive qu'on se demande si des séries qui proviennent d'un problème énumératif, et qui en général sont solutions d'une équation aux différences,
sont aussi solutions d'une équation différentielle, linéaire ou pas. Ceci a pour but d'obtenir
des informations sur les récurrences qui engendrent les séries en question.
On pourra citer à titre d'exemple \cite{BousquetMishna}, \cite{BousquetPetk} et \cite{BostanKauers} et \cite{singerletter}.
\end{itemize}
Pour conclure on se limitera à faire une liste, quasiment en vrac, de quelques résultats en relation avec le sujet principal de ce texte.
\par
Dans \S\ref{sec:Holder}, on a déjà beaucoup parlé d'équations aux différences finie, associées à la translation $x\mapsto x+1$.
En ce qui concerne les équations aux $q$-différences, associées à l'homothétie $x\mapsto qx$,
nous avons d'un côté les résultats de rationalité des séries formelles solutions des systèmes
d'équations aux $q$-différences/différentiels \cite{RamisToulouse} et des systèmes d'équations
aux $q$-différences/$q^\p$-différences \cite{BezivinBoutabaa}. La rationalité des solutions
est aussi étudiée dans \cite{DVInv} et \cite{diviziohardouinqGroth}, par des méthodes arithmétiques
inspirées de la conjecture de Grothendieck sur les $p$-courbures.
De l'autre, on a
le résultat de Ishizaki \cite{Ishizaki} sur l'hypertranscendance des solutions méromorphes des équations
de la forme $y(qx)=a(x)y(x)+b(x)$. Une premier approche galoisienne à ce sujet se trouve dans \cite{hardouincompositio},
suivi par le travail \cite{HardouinSinger}, sur lequel on s'est longuement étendu.
\par
La transcendance différentielle des fonctions de Mahler $f(x)=\sum_{n\geq 0}x^{k^n}$ est étudiée dans \cite{mahler}
et \cite{Loxtonpoorten}. La fonction $f$ est solution de l'équation fonctionnelle $f(x^k)=f(x)-x$.
La question de la transcendance différentielle des solutions de ce type d'équation fonctionnelle est étudiée, toujours
par des méthodes galoisiennes, dans la thèse de P. Nguyen, dont les résultats sont annoncés dans la note
\cite{pierre}. M. Singer a aussi prouvé des résultat dans cette direction \cite{singerletter}.
\par
Pour ce qui concerne les travaux en théorie de Galois paramétrée, il faut signaler que le point de départ a été la théorie
paramétrée des équations différentielles, développée dans \cite{cassisinger}. Le problème inverse a été étudié par M. Singer
\cite{singer2011linear}.
Pour cette théorie on dispose
d'une description du groupe de Galois dans le cas analytique \cite{dreyfus2012density}, dans l'esprit du théorème de densité de Ramis,
et d'un algorithme de Kovacic pour les équations différentielles d'ordre $2$ \cite{dreyfus2011kovacic}.
Signalons aussi l'étude de l'intégrabilité dans \cite{gorchinskiy2012isomonodromic}.
\par
La théorie de Galois paramétrée est liée aux catégories tannakiennes différentielles, introduites par
A. Ovchinnikov \cite{ovchinnikovdifftannakian,ovchinnikovDiffAlgPara} et par M. Kamesky \cite{kamensky2009model}.
Les questions liées à la descente peuvent être traitées via la théorie de Picard-Vessiot \cite{wibmer2011existence}
ou bien l'approche tannakienne \cite{GilletGorchinskyOvchinnikov}.
Par ailleurs, l'analogue de la conjecture de Grothendieck sur les $p$-courbures permet de donner
une caractérisation arithmétique du groupe de Galois intrinsèque \cite{diviziohardouinqGroth}
et de son analogue paramétré \cite{diviziohardouinqMalg} et de le
comparer avec les différentes théorie de Galois dans la littérature \cite{diviziohardouinComp}, en
complétant le travail de comparaison commencé dans \cite{ChatHardouinSinger}.
\par
Il est naturel de se demander si la théorie de Galois peut aider à analyser la transcendance d'une fonction et de ses itérées
par rapport à un automorphisme. Ceci fait l'objet de travaux en cours par l'auteur de ce texte, C. Hardouin et M. Wibmer, d'un côté, et
par A. Ovchinnikov, D. Trushin et M. Wibmer, de l'autre. La géométrie des variétés aux différences étant plus compliquée que la géométrie des variétés différentielles
(au sens de Kolchin), il y a beaucoup de difficultés. Dans cette direction, on citera aussi le travail de M. Kamesky \cite{kamensky2011tannakian}.
\par
De façon un peu surprenante, la théorie de Galois non linéaire \cite{MalgGGF,UmemuranonlinearGalois},
a été développée bien  avant la théorie de Galois paramétrée. Elle a été généralisée au cas des équations aux différences non linéaires
dans \cite{GranierFourier}. La théorie de Galois non linéaire généralise plutôt la théorie de Galois paramétrée que
la théorie de Galois ``classique'' des équations aux différences (voir Corollaire 4.10 dans \cite{diviziohardouinqMalg,diviziohardouinCRAS}).
Les deux papiers \cite{casaleroques,casaleroquesCrelle} mélangent la théorie linéaire des équations fonctionnelles et la théorie
non linéaire de Malgrange pour traiter des problèmes d'intégrabilité.
\par
Enfin, ils existe plusieurs approches différentes à la théorie de Galois des équations aux différences: voir par exemple \cite{SauloyENS} et \cite{andreens}.
Dans le cas particulier des équations aux $q$-différences, les travaux de J.-P. Ramis, J. Sauloy et C. Zhang étudient des questions galoisiennes
d'un point de vue beaucoup plus analytique \cite{RamisSauloySauvageI,RamisSauloySauvageII,ramis2009local}.

\bibliography{difftransc,qG}

\end{document}